\documentclass[12pt]{article}
\usepackage{graphicx,psfrag,latexsym,algorithm,algorithmic,amsmath,amsthm,url,comment,pgfplots}
\includecomment{boringproof}
\excludecomment{comment}

\newcommand{\home}{./}
\newcommand{\myqed}{} 
\newcommand{\proofof}{Proof of~}
\newcommand{\BEAS}{\begin{eqnarray*}}
\newcommand{\EEAS}{\end{eqnarray*}}
\newcommand{\BEA}{\begin{eqnarray}}
\newcommand{\EEA}{\end{eqnarray}}
\newcommand{\BEQ}{\begin{equation}}
\newcommand{\EEQ}{\end{equation}}
\newcommand{\BIT}{\begin{itemize}}
\newcommand{\EIT}{\end{itemize}}
\newcommand{\BNUM}{\begin{enumerate}}
\newcommand{\ENUM}{\end{enumerate}}

\newcommand{\BA}{\begin{array}}
\newcommand{\EA}{\end{array}}

\newcommand{\eg}{{\it e.g.}}
\newcommand{\ie}{{\it i.e.}}

\newcommand{\ones}{\mathbf 1}

\newcommand{\reals}{{\mbox{\bf R}}}

\newcommand{\rank}{\mbox{\textrm{Rank}}}
\newcommand{\nullspace}{{\mathcal {\bf nullspace}}}

\newcommand{\argmin}{\mathop{\rm argmin}}
\newcommand{\epi}{\mathop{\bf epi}}

\newcommand{\card}{\mathop{\bf card}}
\newcommand{\conv}{\mathop{\bf conv}}

\newcommand{\cl}{\mathop{\bf cl}}

\theoremstyle{plain}
\newtheorem{theorem}{Theorem}
\newtheorem{lemma}{Lemma}
\newtheorem{corollary}{Corollary}

\theoremstyle{definition}

\theoremstyle{remark}
\newtheorem{example}{Example}

{\begin{quote}}{\end{quote}}

\makeatletter
\newlength \figwidth
\if@twocolumn
\else
\fi
\makeatother

\makeatletter
\long\def\@makecaption#1#2{
   \vskip 9pt 
   \begin{small}
   \setbox\@tempboxa\hbox{{\bf #1:} #2}
   \ifdim \wd\@tempboxa > 5.5in
        \begin{center}
        \begin{minipage}[t]{5.5in}
        \addtolength{\baselineskip}{-0.95pt}
        {\bf #1:} #2 \par
        \addtolength{\baselineskip}{0.95pt}
        \end{minipage}
        \end{center}
   \else 
	\hbox to\hsize{\hfil\box\@tempboxa\hfil}  
   \fi
   \end{small}\par
}
\makeatother

\newcounter{oursection}

\newcounter{lecture}

\begin{document}

\title{Bounding Duality Gap for Separable Problems with Linear Constraints}
\author{
Madeleine Udell and Stephen Boyd
}
\date{\today}

\maketitle

\begin{abstract}
We consider the problem of minimizing a sum of 
non-convex functions over a compact domain,
subject to linear inequality and equality constraints. 
Approximate solutions can be found by solving a convexified version of the problem,
in which each function in the objective is replaced by its convex envelope.
We propose a randomized algorithm to solve the convexified problem
which finds an $\epsilon$-suboptimal solution to the original problem.
With probability one, $\epsilon$ is bounded by a term 
proportional to the maximal number of active constraints in the problem.
The bound does not depend on the number of variables in the problem or 
the number of terms in the objective.
In contrast to previous related work, 
our proof is constructive, self-contained, and gives a bound that is tight.

\end{abstract}

\section{Problem and results}

\paragraph{The problem.}
We consider the optimization problem
\[
\label{eq-primal}
\tag{$\mathcal{P}$}
\begin{array}{ll}
\mbox{minimize} & f(x) = \sum_{i=1}^n f_i(x_i) \\
\mbox{subject to} & A x \leq b \\
& G x = h,
\end{array}
\]
with variable $x = (x_1, \ldots, x_n) \in \reals^N$,
where $x_i \in \reals^{n_i}$, with $\sum_{i=1}^n n_i = N $. 
There are $m_1$ linear inequality constraints,
so $A \in \reals^{m_1 \times N}$, 
and $m_2$ linear equality constraints, so $G \in \reals^{m_2 \times N}$.
The optimal value of \ref{eq-primal} is denoted $p^\star$.
The objective function terms are lower semi-continuous on their domains:
$f_i: S_i\to \reals$, where $S_i \subset \reals^{n_i}$ is a compact set.
We say that a point $x$ is \emph{feasible} (for \ref{eq-primal}) 
if $Ax \leq b$, $Gx = h$, and 
$x_i \in S_i$, $i=1,\ldots,n$.
We say that \ref{eq-primal} is feasible if there is at least one feasible point.
In what follows, we assume that \ref{eq-primal} is feasible.

Linear inequality or equality constraints that pertain only to a 
single block of variables $x_i$
can be expressed implicitly by modifying $S_i$,
so that $x_i \not \in S_i$
when the constraint is violated. 
Without loss of generality, we assume that this transformation has been 
carried out, so that each of the remaining 
linear equality or inequality constraints 
involves at least two blocks of variables.
This reduces the total number of constraints $m = m_1+m_2$;
we will see later why this is advantageous.
Since each of the linear equality or inequality constraints involves at 
least two blocks of variables, they are called \emph{complicating 
constraints}.
Thus $m$ represents the number of complicating constraints,
and can be interpreted as a measure of difficulty for the problem.

We will state our results in terms of a (possibly) smaller quantity
$\tilde m \leq m$,
which provides a (sometimes) tighter estimate of the number of 
complicating constraints in the problem.
Define the \emph{active set} of inequality constraints at $x$
to be $J(x) = \{j : (Ax - b)_j = 0\}$,
let $\tilde m_1 = \max_x |J(x)|$
be the maximal number of inequality constraints that can be 
simultaneously active,
and let $\tilde m = \tilde m_1 + m_2$ be the number of
(equality and inequality) constraints that can be simultaneously active.

We make no assumptions about the convexity of the functions $f_i$
or the convexity of their domains $S_i$,
so that in general the problem is hard to solve
(and even NP-hard to approximate \cite{udell2013}).

\paragraph{Convex envelope.}
For each $f_i$, we let $\hat{f}_i$ denote its \emph{convex envelope}.
The convex envelope $\hat{f}_i: \conv(S_i) \to \reals$  
is the largest closed convex function
majorized by $f_i$, \ie, $f_i(x) \geq \hat{f}_i(x)$ for all $x$
\cite[Theorem 17.2]{rockafellar1970}. 
When $f_i$ is lower semi-continuous and $S_i$ is compact and nonempty, 
then $\conv(S_i)$ is compact and convex,
and $\hat{f}_i$ is closed, proper, and convex \cite{rockafellar1970}.
In \S \ref{s-examples}, we give a number of examples 
in which we compute $\hat{f}_i$ explicitly.

\paragraph{Nonconvexity of a function.}
Define the \emph{nonconvexity} $\rho(f)$ 
of a function $f: S \to \reals$ to be
\[
\rho(f) = \sup_x (f(x) - \hat{f}(x)),
\]
where for convenience we define a function to be infinite outside of its domain
and interpret $\infty - \infty$ as $0$.
Evidently $\rho(f) \geq 0$, 
and $\rho(f) = 0$ if and only if $f$ is convex and closed.
The nonconvexity $\rho$ is finite if 
$f$ is bounded and lower semi-continuous and $S$ is compact and convex.
For convenience, we assume that the functions $f_i$ are sorted 
in order of decreasing nonconvexity,
so $\rho(f_1) \geq \cdots \geq \rho(f_n)$.

\paragraph{Convexified problem.}
Now, consider replacing each $f_i$ by $\hat{f}_i$
to form a convex problem,
\[
\label{eq-primal-conv}
\tag{$\hat{\mathcal{P}}$}
\begin{array}{ll}
\mbox{minimize} & \hat{f}(x) = \sum_{i=1}^n \hat{f}_i(x_i) \\
\mbox{subject to} & A x \leq b \\
& G x = h,
\end{array}
\]
with optimal value $\hat{p}$.
This problem is convex; if we can efficiently
evaluate $\hat{f}$ and a subgradient (or derivative, if 
$\hat{f}$ is differentiable), then the problem is easily solved using standard methods
for nonlinear convex optimization.
Furthermore, \ref{eq-primal-conv} is feasible as long as \ref{eq-primal} is feasible. 
Evidently $\hat p \leq p^\star$; that is, the optimal value of the 
convexified problem is a lower bound on the optimal value of the original
problem.
We would like to know when a solution to \ref{eq-primal-conv} approximately
solves \ref{eq-primal}. 


Our first result is the following:

\begin{theorem}
\label{thm-bound}
There exists a solution $x^\star$ of \ref{eq-primal-conv} such that
\[
\hat p = \hat f(x^\star) \leq f(x^\star) \leq \hat{p} + \sum_{i=1}^{\min(\tilde m,n)} \rho(f_i).
\]
\end{theorem}

Since $p^\star \leq f(x^\star)$ and $\hat p \leq p^\star$, Theorem~\ref{thm-bound}
implies that 
\[
p^\star \leq f(x^\star) \leq p^\star + \sum_{i=1}^{\min(\tilde m,n)} \rho(f_i).
\]
In other words, there is a solution of the convexified problem that is
$\epsilon$-suboptimal for the original problem, with
$\epsilon = \sum_{i=1}^{\min(\tilde m,n)} \rho(f_i)$.
It is not true (as we show in \S\ref{s-counterexample}) 
that all solutions of the convexified problem are
$\epsilon$-suboptimal.

Theorem~\ref{thm-bound} shows that if the objective function terms are not too nonconvex,
and the number of (active) constraints is not too large, then the 
convexified problem has a solution that is not too suboptimal for the 
original problem.
This theorem is similar to a number of results previously in the literature;
for example, it can be derived from 
the well-known Shapley-Folkman theorem \cite{Starr1969}.
A looser version of this theorem may be obtained from 
the bound on the duality gap given in \cite{Aubin1976}.

Theorem~\ref{thm-bound} also implies a bound on the duality gap 
for problems with separable objectives.
Let 
\[
L (x,\lambda,\mu) = \sum_{i=1}^n f_i(x_i) + \lambda^T (A x - b) + 
\mu^T(G x - h)
\] 
be the Lagrangian of \ref{eq-primal} with dual variables
$\lambda$ and $\mu$,
and define the (Lagrange) dual problem to \ref{eq-primal},
\[
\label{eq-dual}
\tag{$\mathcal{D}$}
\begin{array}{ll}
\mbox{maximize} & \inf_x \mathcal L (x,\lambda,\mu)\\
\mbox{subject to} &  \lambda \geq 0,
\end{array}
\]
with optimal value $g^\star$.
The convexified problem \ref{eq-primal-conv} is the dual of \ref{eq-dual}.
(See Appendix~\ref{s-dual-dual} for a derivation.)
Since \ref{eq-primal-conv} is convex and feasible, with only linear constraints, 
strong duality holds by the refined Slater's constraint qualification \cite[\S 5.2.3]{boyd2004}. 
(For a proof, see \cite[p. 277]{rockafellar1970}.)
Hence the maximum of the dual problem is attained,
\ie, $g^\star = \hat p$
and $\inf_x \mathcal L (x,\lambda^\star) = g^\star$ 
for some $\lambda^\star \geq 0$ .
The bound from Theorem \ref{thm-bound} thus implies
 \[
p^\star - g^\star \leq \sum_{i=1}^{\min (\tilde m,n)} \rho(f_i).
\]

What is not clear in other related work is how to construct a feasible solution 
that satisfies this bound.
This observation leads us to the main contribution of this paper:
a constructive version of Theorem~\ref{thm-bound}.

\begin{theorem}
\label{thm-construct}
Let $w \in \reals^N$ be a random variable with uniform distribution
on the unit sphere.
Now consider the feasible convex problem
\BEQ
\label{eq-lp}
\tag{$\mathcal{R}$}
\begin{array}{ll}
\mbox{minimize} & w^T x\\
\mbox{subject to} & Ax \leq b \\
                  & Gx = h \\
                  & \hat{f}(x) \leq \hat p.
\end{array}
\EEQ
Then with probability one, \ref{eq-lp} has a unique solution $x^\star$
which satisfies the inequality of Theorem~\ref{thm-bound},
\[
f(x^\star) \leq \hat{p} + \sum_{i=1}^{\min(m,n)} \rho(f_i),
\]
\ie, $x^\star$ is $\epsilon$-suboptimal for the original problem \ref{eq-primal}.
\end{theorem}

The randomized problem~\ref{eq-lp} has a simple interpretation. 
Any feasible point $x$ for \ref{eq-lp} is feasible for \ref{eq-primal-conv},
and the constraint $\hat{f}(x) \leq \hat{p}$ is satisfied with equality.
That is, \ref{eq-lp} minimizes a random linear function over
the optimal set of \ref{eq-primal-conv}. 
Theorem \ref{thm-construct} tells
us that this construction yields (almost surely) an
$\epsilon$-suboptimal solution of \ref{eq-primal}.

We give a self-contained proof of both of these theorems in \S \ref{s-proof}.


\section{Discussion}

In this section 
we show that the bound in Theorem \ref{thm-bound} is tight,
and that finding extreme points of the optimal set is essential to achieving the bound.
In these examples, $\tilde m = m$.

\begin{example}[The bound is tight.]
Consider the problem
\BEQ
\label{eq-rand-primal}
\begin{array}{ll}
\mbox{minimize} & \sum_{i=1}^n g (x_i) \\
\mbox{subject to} & \sum_{i=1}^n x_i \leq B,
\end{array}
\EEQ
with $g: [0,1] \to \reals$ defined as
\[
g(x) = \left\{
        \begin{array}{ll}
            1 & \quad 0 \leq x < 1 \\
            0 & \quad x = 1.
        \end{array}
    \right. 
\]
The convex envelope $\hat{g}:[0,1] \to \reals$ of $g$ is given by
$\hat{g}(x) = 1-x$, with $\rho(g) = 1$.
The convexified problem \ref{eq-primal-conv} corresponding 
to (\ref{eq-rand-primal}) is 
\BEQ
\label{eq-rand-conv}
\begin{array}{ll}
\mbox{minimize} & \sum_{i=1}^n \hat{g} (x_i) \\
\mbox{subject to} & \sum_{i=1}^n x_i \leq B \\
& 0 \leq x.
\end{array}
\EEQ

Any $x^\star$ satisfying $0 \leq x^\star \leq 1$ and 
$\sum_{i=1}^n x^\star_i = B$ is optimal for 
the convexified problem (\ref{eq-rand-conv}), with value $\hat{p} = n - B$.
If $B < 1$, then the optimal value of (\ref{eq-rand-primal})
is $p^\star = n$.
Since (\ref{eq-rand-primal}) has only one constraint, 
the bound from Theorem \ref{thm-bound} applied to this problem gives 
\[
n = p^\star \leq \sum_{i=1}^n g(x^\star_i) \leq \hat{p} + \rho(g) = n - B + 1.
\]
Letting $B \to 1$, we see that the bound is tight.
\end{example}

\begin{example}[Find the extreme points.]
\label{s-counterexample}
Not all solutions to the convexified problem satisfy 
the bound from Theorem \ref{thm-bound}.
As we show in \S\ref{s-proofs}, the value of the convex envelope 
at the extreme points of the optimal set for the convexified problem will be 
provably close to the value of the original function,
whereas the difference between these values on the interior of the optimal set
may be arbitrarily large.

For example, suppose $n-1 < B < n$ in the problem defined above.
As before, the optimal set for the convexified problem (\ref{eq-rand-conv}) is 
\[
M = \{x: \textstyle \sum_{i=1}^n x_i = B, ~x_i \geq 0,~ i=1,\ldots,n\}. 
\]

Consider $\hat{x}\in M$ with $\hat{x}_i = B/n$, $i=1,\ldots,n$, 
which is optimal for the convexified problem (\ref{eq-rand-conv}). 
This $\hat{x}$ does \emph{not} obey the bound in Theorem \ref{thm-bound};
indeed, the suboptimality of $\hat{x}$ grows linearly with $n$.
With this $\hat x$, the left hand side of the inequality in Theorem \ref{thm-bound} 
is $\sum_{i=1}^n g(\hat{x}_i) = n$,
while the right hand side $\hat{p} + \rho(g) = n - B + 1 < 2$ is much smaller.

On the other hand, $x^\star\in M$ defined by
\[
x^\star_i = \left\{ \begin{array}{ll} 
1 & i=1,\ldots,n-1 \\
B - (n-1) & i=n,
\end{array} \right. 
\]
which is an extreme point of the optimal set for the convexified problem,
is optimal for the original problem as well.
That is, $x^\star$ is an extreme point of $M$ that satisfies Theorem \ref{thm-bound}
with equality.
\end{example}

\begin{example}[Nonconvex feasible set.]
For an even simpler example, consider the following univariate problem with no constraints.
Let $S = \{0\}\cup\{1\}$ with $f(x)=0$ for $x\in S$.
Then $\hat{f}: [0,1] \to \{0\}$, so the optimal set for the convexified problem consists of the entire interval $[0,1]$.
But $\hat{x} = 1/2 \in M$ is not feasible for the original problem;
its value according to the original objective is thus infinitely worse than 
the value guaranteed by Theorem \ref{thm-bound}.
On the other hand, $x=0$ and $x=1$, the extreme points of the optimal set for the convexified problem,
are indeed optimal for the original problem.
\end{example}

\section{Related work}

Our proof is very closely related to the Shapley-Folkman theorem \cite{Starr1969},
which states, roughly, that the nonconvexity of
the average of a number of nonconvex sets decreases with the number of sets.
In optimization, the analogous statement is that optimizing the average of a number of functions
is not too different from optimizing the average of the convex envelopes of those functions,
and the difference decreases with the number of functions.
However, we note that using the Shapley-Folkman theorem directly,
rather than its optimization analogue,
results in a bound that is slightly worse. 
For example, the Shapley-Folkman theorem has previously been used 
by Aubin and Ekeland in \cite{Aubin1976} 
to prove a bound on the duality gap.
The bound they present,
\[
p^\star - d^\star \leq \min(m+1,n) \rho(f_1),
\]
is not tight; our bound, which is tight, 
is smaller by a factor of $\tilde m/(m+1)$.

The Shapley-Folkman theorem has found uses in a number of applications within 
optimization. For example, Bertsekas et al. \cite{bertsekas1983} used the theorem
to solve a unit commitment problem in electrical power system scheduling, 
in which case the terms in the objective are univariate.
The Shapley-Folkman theorem and its relation to a bound on the duality gap
also have found applications in integer programming \cite{vujanic2014}.
While we restrict ourselves here only to nonconvex \emph{objectives},
many authors \cite{bertsekas1982,bertsekas1999,lemarechal2001} have studied 
convexifications of separable constraints as well.
A more modern treatment, in the case of linear programs, is given in \cite{bertsekas2009}.

The use of randomization to find approximate solutions to nonconvex problems
is widespread, and often startlingly successful \cite{motwani1995,goemans1995}.
The usual approach is to solve a convex problem to find an optimal probability distribution
over possible solutions; sampling from the distribution and rounding yields the 
desired result.
By contrast, our approach uses randomization only 
to explore the geometry of the optimal set \cite{skaf2010}. 
We rely on the insight that extremal points of the epigraph of the
convex envelope are likely to be closer in value to the original function,
and use randomization simply to reach these points.
Randomization allows us to find ``simplex-style'' corner points 
of the optimal set as solutions, 
rather than accepting interior points of the set.

Our procedure for finding an extreme point is closely related to
the idea of \emph{purifying} a solution returned by, \eg, an interior point solver to
obtain an extremal solution. 
One fixes an active set of inequality constraints that hold with equality at a given point,
and solves (\ref{eq-lp}) subject to the additional constraint that all inequality constraints in the active set
continue to hold with equality, and then iterates this procedure until the set of active constraints
completely determines the solution.
It is easy to see that at each iteration at least one constraint is added to the active set.
Hence the procedure converges to an extreme point
in no more than $\tilde m$ iterations.
In contrast, our proof shows that the method finds an extreme point with probability 1 in a single iteration,
without fixing an active set beforehand.

The notion that extreme points of the solution set of a convex problem have 
particularly nice properties is pervasive in the literature.
The extreme points produced by solving \ref{eq-lp} are simply
\emph{basic feasible solutions}, familiar from the analysis of the simplex method,
whenever the functions $f_i$ are univariate, \ie, $n_i = 1$, $i=1,\ldots,n$.
Other uses of extreme points abound:
for example, Anderson and Lewis \cite{anderson1989} propose a simplex-style method for 
semi-infinite programming that proceeds by finding extreme points of the 
feasible set;
and Barvinok \cite{barvinok1995,barvinok2002} and 
Pataki \cite{pataki1996,pataki1998} examine
the extreme points of an affine section of the semidefinite cone
to provide bounds on the rank of solutions to semidefinite programs.










\section{Constructing the convex envelope}
\label{s-envelopes}

In general, the convex envelope of a function can be hard to compute.
But in many special cases, we can efficiently construct the convex envelope
or a close approximation to it.
The problem of computing convex lower bounds on general nonconvex functions 
has been extensively studied in the global optimization community:
see, eg, \cite{horst2000} for a general introduction and 
\cite{tawarmalani2002} for a more sophisticated treatment.
In this section, we give a few examples illustrating how to construct the convex envelope
for a number of interesting functions and classes of functions.

\paragraph{Sigmoidal functions.}
A continuous function $f:[l,u]\to \reals$
is defined to be \emph{sigmoidal} if it is either convex, concave, or
convex for $x \leq z \in [l,u]$ and concave for $x \geq z$.
For a sigmoidal function,
the convex envelope is particularly easy to calculate \cite{udell2013}.
We can write $\hat{f}$ of $f$ piecewise as
\[
\hat{f}(x) = 
\left\{
\begin{array}{ll}
f(x) & l \leq x \leq w \\
f(w) + \frac{f(u) - f(w)}{u-w}(x-w) & w \leq x \leq u \\
\end{array}
\right.
\] 
for some appropriate $w\leq z$. 
If $f$ is differentiable, then $f'(w) = \frac{f(u) - f(w)}{u-w}$;
in general, $\frac{f(u) - f(w)}{u-w}$ is a subgradient of $f$ at $w$.
The point $w$ can easily be found by bisection: 
if $x > w$, then the line from $(x,f(x))$ to $(u,f(u))$ crosses the graph 
of $f$ at $x$; 
if $x < w$, it crosses in the opposite direction.

\paragraph{Univariate functions.}
If the inflection points of the univariate function are known,
then the convex envelope may be calculated by iterating the construction
given above for the case of sigmoidal functions.

\paragraph{Analytically.} Occasionally the convex envelope may be calculated 
analytically.
For example, convex envelopes of multilinear functions
on the unit cube are polyhedral (piecewise linear), 
and can be calculated using an analytical formula
given in \cite{rikun1997}.
A few examples of analytically tractable convex envelopes are presented in Table \ref{t-examples}.
In the table, $\hat{f}: \conv (S) \to \reals $ is the convex envelope of $f: S \to \reals$, 
and $\rho (f)$ gives the nonconvexity of $f$.
We employ the following standard notation: 
$\card(x)$ denotes the cardinality (number of nonzeros) of the vector $x$; 
the spectral norm (maximum singular value) is written as $\|M\|$,
and its dual, the nuclear norm (sum of singular values) is written as $\|M\|_*$.


\paragraph{Via differential equations.} The convex envelope of a function
can also be written as the solution to a certain nonlinear partial differential
equation \cite{oberman2007},
and hence may be calculated numerically using the standard machinery of
numerical partial differential equations \cite{oberman2008}.

\begin{table}[htb!]
\label{t-examples}
\caption{Examples of convex envelopes.}
\begin{tabular}{c|c|c|c}
$S$ & $f(x)$ & $\hat{f}(x)$ & $\rho(f)$ \\
\hline
$[0,1]^2$ & $\min(x,y)$ & $(x+y-1)_+$ & 1/2\\
$[0,1]^2$ & $xy$ & $(x+y-1)_+$ & 1/4\\
$[0,1]^n$ & $\min(x)$ & $(\sum_{i=1}^n x_i - (n-1) )_+$ & $\frac{n-1}{n}$\\
$[0,1]^n$ & $\prod_{i=1}^n x_i$ & $(\sum_{i=1}^n x_i - (n-1) )_+$ & $(\frac{n-1}{n})^n$\\
$[-1,1]^n$ & $\card(x)$ & $\|x\|_1$ & $n$\\
$\{M \in \reals^{k \times n}: \|M\| \leq 1\}$ & $\rank(M)$ & $\|M\|_*$ & $n$\\
\end{tabular}
\end{table}

\section{Examples} \label{s-examples}

\paragraph{Resource allocation.}
An agent wishes to allocate resources to a collection of projects $i=1,\ldots,n$.
For example, the agent might be bidding on a number of different auctions,
or allocating human and capital resources to a number of risky projects.
There are $m$ different resources to be allocated to the projects, with each project $i$
receiving a non-negative quantity $x_{ij}$ of resource $j$.
The probability that project $i$ will succeed is modeled as $f_i(x_i)$,
and its value to the agent, if the project is successful, is given by $v_i$.
The agent has access to a quantity $c_j$ of resource $j$, $j=1,\ldots,m$.
An allocation is feasible if $\sum_{i=1}^n x_{ij} \leq c_j$, $j=1,\ldots,m$.
The agent seeks to maximize the expected value of the successful projects by solving
\[
\begin{array}{ll}
\mbox{maximize} & \sum_{i=1}^n v_i f_i(x_i) \\
\mbox{subject to} & \sum_{i=1}^n x_{ij} \leq c_j,\quad j=1,\ldots,m\\
& x\geq 0.
\end{array}
\]
To conform to our notation in the rest of this paper, we write this as a minimization problem, 
\[
\begin{array}{ll}
\mbox{minimize} & \sum_{i=1}^n -v_i f_i (x_i) \\
\mbox{subject to} & \sum_{i=1}^n x_{ij} \leq c_j,\quad j=1,\ldots,m\\
& x\geq 0.
\end{array}
\]
Here, there are $m$ complicating constraint connecting the variables.
Hence the bound from Theorem \ref{thm-bound} guarantees that
$|\hat{p} - p^\star| \leq \sum_{i=1}^{\min(m,n)} \rho(f_i)$.
If $p_i: \reals \to [0,1]$ is a probability, then $\rho (-v_i p_i) \leq v_i$.
For example, if there is only one resource ($m=1$), 
the bound tells us that we can find a solution $x$ by solving 
the convex problem \ref{eq-lp} whose value differs from the true optimum $p^\star$
by no more than $\max_i v_i$, regardless of the number of projects $n$.

\paragraph{Flow and admission control.}

A set of flows pass through a network over given paths of links or edges;
the goal is to maximize the total utility while respecting
the capacity of the links.
Let $x_i$ denote the level of each flow $i=1,\ldots,n$ and
$u_i(x_i)$ the utility derived from that flow.
Each link $j$, $j=1,\ldots,m$, is shared by the flows $i \in S_j$,
and can accomodate up to a total of $c_j$ units of flow.
The flow routes are defined by a matrix $A \in \reals^{m \times n}$ 
mapping flows onto links,
with entries $a_{ji}$, $i=1,\ldots,n$, $j=1,\ldots,m$.
When flows are not split, \ie, they follow simple paths,
we have $a_{ij} = 1$ when flow $i$ pass over link $j$, and
$a_{ij}=0$ otherwise.  But it is also possible to split a flow across
multiple edges, in which case the entries $a_{ij}$ can take other 
values.
The goal is to maximize the total utility of the flows,
subject to the resource constraint,
\BEQ
\label{eq-flow}
\begin{array}{ll}
\mbox{maximize} & \sum_{i=1}^n u_i(x_i) \\
\mbox{subject to} & Ax \leq c\\
& x\geq 0.
\end{array}
\EEQ

The utility function is often modelled by a bounded
function, such as a sigmoidal 
function \cite{udell2013,Fazel2005}. 
As an extreme case, we can consider utilities of the form
\[
u(x) = \left\{
        \begin{array}{ll}
            0 & \quad x < 1 \\
            1 & \quad x \geq 1.
        \end{array}
    \right. 
\] 
Thus each flow has value $1$ when its level is at least $1$, and no value
otherwise.  In this case, the problem is to determine choose the 
subset of flows, of maximum cardinality, that the network can handle.
(This problem is also called admission control, since we are deciding which 
flows to admit to the network.)

We can replace this problem with an equivalent minimization problem to facilitate
the use of Theorem~\ref{thm-bound}. 
Let $f_i(x) = -u_i(x)$. Then we minimize the negative utility of the flows by solving
\[
\begin{array}{ll}
\mbox{minimize} & \sum_{i=1}^n u_i(x_i) \\
\mbox{subject to} & Ax \leq c\\
& x\geq 0.
\end{array}
\]
Suppose $f_i$ is bounded for every $i$, so that $\max_i \rho(f_i) \leq R$.
Then the bound from Theorem \ref{thm-bound} guarantees that we can quickly find a solution
$p^\star - \hat{p} \leq mR$.
In a situation with many flows but only a modest number of links,
the solution given by solving \ref{eq-lp} may be very close to optimal.



\section{Proofs} \label{s-proofs}

To simplify the proofs in this section, 
we suppose without loss of generality
that the problem has only inequality constraints;
the mathematical argument with equality constraints is exactly the same.
Merely note below in Lemma~\ref{thm-extreme-active} that equality constraints are always active.
We let $A = [A_1 \cdots A_n]$ with $A_i \in \reals^{m \times n_i}$,
so $Ax = \sum_i A_i x_i$. As before, $N = \sum_{i=1}^n n_i$.

\subsection{Definitions}
First, we review some basic definitions from convex analysis
(see \cite{rockafellar1970,lemarechal2001} for more details).

The \emph{epigraph} of a function $f$ is the set of points lying above the graph of $f$,
\[
\epi(f) = \{(x,t) : t \geq f(x)\}.
\]
The \emph{convex hull} of a set $S$ is the set of points that can be
written as a convex combination of other points in the set,
\[
\conv (S) = \left\{ \textstyle \sum_j \theta_j x_j : \theta_j \geq 0, ~
x_j \in S, ~ \sum_j \theta_j = 1 \right\}.
\]

An \emph{exposed face} $F$ of a convex set $C$ is a set of points 
optimizing a linear functional over that set,
\[
F = \argmin_{x \in C} c^T x,
\]
for some $c \in \reals^n$. The vector $c$ is called a normal vector to the face.
We will use the fact that every exposed face is a \emph{face}:
a convex set $F \subset C$ for which every (closed) line segment in $C$ with a relative interior point in $F$ has both endpoints in $F$. 
Not all faces are exposed; but our analysis will not make use of this distinction. 

An \emph{extreme point} of a convex set is a point that cannot be written as a convex combination
of other points in the set.
It is easy to see that a zero-dimensional exposed face of a convex set is an extreme point,
and that any extreme point defines a zero-dimensional exposed face of a convex set \cite{rockafellar1970}.

\subsection{Main lemmas}
\label{s-proof}
Our analysis relies on two main lemmas.
Lemma \ref{thm-extreme-same} tells us that at the extreme points of an exposed face of $\epi(\hat f)$,
the values of $f$ and $\hat{f}$ are the same.
Lemma \ref{thm-extreme-lp} tells us that (with probability one) we can find a point that is 
extreme in $\epi(\hat f_i)$ for most $i$, 
and feasible, by solving a randomized convex program.
We then combine these two lemmas to prove Theorem \ref{thm-construct} and,
as a consequence, Theorem \ref{thm-bound}.

We use two other technical lemmas as ingredients in the proofs of the two main lemmas.
Lemma \ref{thm-closed} gives conditions under which the convex hull of the epigraph
of a function is closed,
and Corollary \ref{thm-well-posed-linear-compact} states that the maximum of a random
linear functional over a compact set is unique with probability one.
Their statements and proofs can be found in Appendix \ref{s-closed} and 
Appendix \ref{s-well-posed} respectively.

We begin by finding a set of points where $f$ and $\hat{f}$ agree.
\begin{lemma}
\label{thm-extreme-same}
Let $S \subset \reals^n$ be a compact set, and 
let $f: S \to \reals$ be lower semi-continuous on $S$,
with convex envelope $\hat{f}: \conv (S) \to \reals$.
Let $c\in \reals^n$ be a given vector.
If $x$ is extreme in the set $\argmin (\hat{f}(x) + c^T x)$,
then $x \in S$ and $f(x) = \hat{f}(x)$.
\end{lemma}

\begin{proof}
The vector $c$ defines an (exposed) face 
$\{(y, \hat{f}(y)) \mid y \in \argmin (\hat{f}(x) + c^T x)\}$
of $\epi(\hat{f})$.
If $x$ is extreme in $\argmin (\hat{f}(x) + c^T x)$,
then $(x,\hat{f}(x))$ is extreme in $\epi(\hat{f})$ \cite[p. 163]{rockafellar1970}.

It is easy to see geometrically that every extreme point of $\epi(\hat{f})$
is a point in $\epi(f)$.
Formally, recall that the convex envelope satisfies 
$\epi(\hat{f}) = \cl (\conv (\epi(f)))$ \cite[cor. 12.1.1]{rockafellar1970}.
Then use Lemma \ref{thm-closed} (see Appendix \ref{s-closed}),
which states that the $\conv (\epi(f))$ is closed if $S$ is compact and 
$f$ is lower semi-continuous, to see that $\cl (\conv (\epi(f))) = \conv (\epi(f))$.
Thus every extreme point of $\epi(\hat{f})$
is a point in $\epi(f)$ \cite[cor. 18.3.1]{rockafellar1970}.

So $(x,\hat{f}(x)) \in \epi(f)$, and hence $x \in S$ and $\hat{f}(x) \geq f(x)$.
But $\hat{f}$ is the convex envelope of $f$, so $\hat{f}(x) \leq f(x)$.
Thus $\hat{f}(x) = f(x)$.
\myqed
\end{proof}

Now we show that a solution to a randomized convex program 
finds a point that is extreme for most subvectors $x_i$ of $x$.
\begin{lemma}
\label{thm-extreme-lp}
Let $M_i \in \reals^{n_i}$, $i=1,\ldots,n$, be given compact convex sets, and let
$A \in R^{m \times N}$ with $N = \sum_{i=1}^n n_i$.
Choose $w \in \reals^{N}$ uniformly at random on the unit sphere,
and consider the convex program
\BEA
\label{eq-lp-set}
\begin{array}{ll}
\mbox{minimize} & w^T x \\
\mbox{subject to} & A x \leq b \\
                  & x_i \in M_i, \quad i=1,\ldots,n.
\end{array}
\EEA
Almost surely, the solution $x$ to problem~(\ref{eq-lp-set}) is unique.
For all but at most $\tilde m$ indices $i$, $x_i$ is an extreme point of $M_i$.
\end{lemma}

To prove Lemma~\ref{thm-extreme-lp}, we will prove the following stronger lemma.
Lemma~\ref{thm-extreme-lp} follows as a corollary, 
since $\tilde m$ bounds the number of simultaneously active constraints.

\begin{lemma}
\label{thm-extreme-active}
Let $M_i \in \reals^{n_i}$, $i=1,\ldots,n$, be given compact convex sets, and let
$A \in R^{m \times N}$ with $N = \sum_{i=1}^n n_i$.
Choose $w \in \reals^{N}$ uniformly at random on the unit sphere,
and consider the convex program
\BEA
\label{eq-lp-set-active}
\begin{array}{ll}
\mbox{minimize} & w^T x \\
\mbox{subject to} & A x \leq b \\
                  & x_i \in M_i, \quad i=1,\ldots,n.
\end{array}
\EEA
Almost surely, the solution $x$ to problem~(\ref{eq-lp-set-active}) is unique.
Let $J = \{j : (Ax - b)_j = 0\}$ be the set of active constraints at $x$.
For all but at most $|J|$ indices $i$, $x_i$ is an extreme point of $M_i$.
\end{lemma}

\begin{proof}[\proofof Lemma~\ref{thm-extreme-active}]
By Corollary \ref{thm-well-posed-linear-compact} (see Appendix~\ref{s-well-posed}), the minimum of a random
linear functional over a compact set is unique with probability one.
Hence we may suppose problem~(\ref{eq-lp-set}) has a unique solution, which we call $x$,
with probability one.
Define $M = M_1 \times \cdots \times M_n$ to be the Cartesian product of the sets $M_i$.
Let $F$ be a minimal face of $M$ containing $x$, 
and let $B \subset F \subseteq M$ be a ball in its relative interior.
If $x$ is on the boundary of $M$, then $\dim(B) < N$.

Let $A_J$ be a matrix consisting of those rows of $A$ with indices in $J$,
and define the minimal distance to any non-active constraint
\[
\delta = \inf_{j\in J^C} \inf_{y: (Ay - b)_j = 0} \|x - y\|.
\]
Let $D = (x + \nullspace(A_J)) \cap \mathcal{B}(x,\delta)$ where
$\mathcal{B}(x,\delta)$ is an open ball around $x$ with radius $\delta$.
With this definition, any $y \in D$ satisfies the constraints $Ay - b$ 
with the same active set $J$:
$(Ay - b)_j = 0$ for every $j \in J$, and $(Ay - b)_j > 0 $ for every $j \in J^C$.
Note that $\dim(D) = \dim(\nullspace(A_J)) = N - |J|$.

Now we will show $B \cap D = \{x\}$.
By way of contradiction, consider $y \in B \cap D$, $y \ne x$.
Every such $y$ is feasible for problem~(\ref{eq-lp-set}).
The random vector $w$ must be orthogonal to $y-x$,
for otherwise the solution to problem~(\ref{eq-lp-set}) could not occur at the center $x$ 
of the feasible ball $B$.
On the other hand, if $w$ is orthogonal to $y-x$, 
then $y$ is a solution to problem~(\ref{eq-lp-set}).
But the solution $x$ is unique, so it must be that $B \cap D = \{x\}$.
That is, $B$ intersects the $(N-|J|)$-dimensional set $D$ at a single point.
This bounds the dimension of $B$: $\dim(B) + \dim(D) \leq N$,
so $\dim(B) \leq |J|$.

Furthermore, $\dim(B)$ bounds the number of 
subvectors $x_i$ of $x$ that are not extreme in $M_i$.
Let 
\[
\Omega = \left\{i \in \{1,\ldots,n\} : x_i \mbox{ is not extreme in } M_i\right\}.
\]
For $i \in \Omega$, $x_i$ lies on a face of $M_i$ with dimension greater than zero.
Hence $B$ contains a point $y^i$ that differs from $x$ only on the $i$th coordinate block. 
Consider the set $Y = \{y^i : i \in \Omega\} \subset B$.
The vectors $y^i - x$ for $i \in \Omega$ are mutually orthogonal, so
$|\Omega| = \dim(\conv(Y)) \leq \dim(B)$.
The argument in the last paragraph showed $\dim(B) \leq |J|$,
and so we can bound the number of subvectors that are not extreme $|\Omega| \leq |J|$.

Thus almost surely, the solution to problem~(\ref{eq-lp-set}) is unique, and
no more than $|J|$ subvectors $x_i$ of the solution $x$
are not at extreme points.
\myqed
\end{proof}


\subsection{Main theorems} 
We are now ready to prove the main theorems, using the previous lemmas.

\begin{proof}[\proofof Theorem~\ref{thm-construct}]
By Lemma \ref{thm-extreme-lp}, the solution $x^\star$ to 
\ref{eq-lp} is unique with probability 1.
Every point in the feasible set for \ref{eq-lp} is optimal for 
\ref{eq-primal-conv}, so in particular, $x^\star$ solves \ref{eq-primal-conv}.
Pick $\lambda^\star \geq 0$ so that $(x^\star,\lambda^\star)$ 
form an optimal primal-dual pair for 
the primal-dual pair (\ref{eq-primal-conv}, \ref{eq-dual}).
Note that by complementary slackness, 
any optimal point $x$ for \ref{eq-primal-conv} 
(and so any feasible point for \ref{eq-lp}) satisfies 
$\lambda^{\star T} (A x - b) = 0$.

Now consider the problem
\BEQ
\label{eq-lp-dual}
\begin{array}{ll}
\mbox{minimize} & w^T x\\
\mbox{subject to} & Ax \leq b \\
                  & \hat{f}(x) - \lambda^{^\star T} A x \leq \hat{p} - \lambda^{^\star T} A x^\star,
\end{array}
\EEQ
where, compared to \ref{eq-lp}, 
we have subtracted $\lambda^{\star T} A x$ and $\lambda^{\star T} A x^\star$ 
from the two sides of the inequality $\hat{f}(x) \leq \hat p$. 

In fact, the feasible set of \ref{eq-lp} is the same as that of problem~(\ref{eq-lp-dual}). 
By complementary slackness, $\lambda^{\star T} A x^\star = \lambda^{\star T} b$,
so the last inequality constraint in problem~(\ref{eq-lp-dual}) can be rewritten as
\[
\hat{f}(x) - \lambda^{^\star T} (A x - b) \leq \hat{p}.
\]
Since $\lambda^\star \geq 0$, 
and $A x - b \leq 0$ on the feasible set of problem~(\ref{eq-lp-dual}),
we have $- \lambda^{^\star T} (A x - b) \geq 0$. 
Hence any $x$ feasible for problem~(\ref{eq-lp-dual}) satisfies
\[
\hat{f}(x) \leq \hat{p},
\]
and so satisfies the constraints of \ref{eq-lp}.
Conversely, any feasible point for \ref{eq-lp} has $\lambda^{^\star T} (A x - b)=0$
by complementary slackness,
so it is also feasible for problem~(\ref{eq-lp-dual}).
Since the feasible sets are the same and the objectives are the same,
the solution to \ref{eq-lp} must also be the same as that of problem~(\ref{eq-lp-dual}). 

Define
\BEAS
M &=& \argmin_x \left( \textstyle \sum_{i=1}^n \hat{f}_i(x_i) - 
\lambda^{\star T} (A x - b)\right) \\
&=& 
\argmin_x \textstyle \sum_{i=1}^n \left( \hat{f}_i(x_i) - 
\lambda^{\star T} A_i x_i \right) - \lambda^{\star T} b.
\EEAS
The function defining the set $M$ is separable.
Hence $M = M_1 \times \cdots \times M_n$, where
\[
M_i = \argmin_{x_i} \left(\hat{f}_i(x_i) - \lambda^{\star T} A_i x_i \right).
\]
The set $M_i$ is compact and convex: 
it is bounded, since the domain of $\hat f_i$, $\conv(S_i)$, is bounded; 
it is closed, since $\epi(\hat{f}_i)$ is closed;
and it is convex, since $\epi(\hat{f}_i)$ is convex.
So the $M_i$ satisfy the conditions for Lemma~\ref{thm-extreme-lp}.

By Lemma \ref{thm-extreme-lp}, the solution $x^\star$ to problem~(\ref{eq-lp-dual}) is unique
and lies at an extreme point of $M_i$ for all but (at most) $\tilde m$ of the coordinate blocks $i$
(with probability one).
By Lemma \ref{thm-extreme-same}, extreme points $x_i$ of $M_i$ 
satisfy $f_i(x_i)=\hat{f}_i(x_i)$, 
so $f_i(\hat x_i) > \hat{f}_i(\hat x_i)$ for 
no more than $\tilde m$ of the coordinate blocks $i$.
On those blocks $i$ where $\hat x_i$ is not extreme, 
it is still true that $f_i(\hat x_i) - \hat{f}_i(\hat x_i) \leq \rho(f_i)$.
Hence
\[
0 \leq \sum_{i=1}^n f_i(x^\star_i) - p^\star
= \sum_{i=1}^n \left( f_i(x^\star_i) - \hat{f}_i(x^\star_i)\right)
\leq  \sum_{i=1}^{\min(\tilde m,n)} \rho(f_i).
\]
\myqed
\end{proof}

\begin{proof}[\proofof Theorem \ref{thm-bound}]
Since a point satisfying the bound in Theorem \ref{thm-bound} 
can be found almost surely by minimizing
a random linear function over $M$, it follows that such a point exists.
\myqed
\end{proof}

\section{Numerical example}

We now present a numerical example to demonstrate the performance of the algorithm
implied by the proof; namely, of finding an extreme point of the convexified problem
to serve as an approximate solution to the original problem. 
This problem is not large,
and is easy to solve using many methods. 
Our purpose in presenting the example
is merely to give some intuition for the utility of 
finding an extreme point of the solution set of the convexified problem,
rather than an arbitrary solution.

\paragraph{Investment problem.}

Consider the following investment problem.
Each variable $x_i \in \reals$ represents the allocation of capital to project $i$.
The probability that a project will fail is given by $f(x_i)$.

Entry $a_{ij}$ of the matrix $A \in \reals^ {m \times n}$ 
represents the exposure of project $i$ to sector $j$ of the economy.
The budget for projects in each sector is given by the vector $b \in \reals^m$.
The constraint $Ax \leq b$ then prevents overexposure to any given sector.

The problem of minimizing the expected number of failed projects 
subject to these constraints can be written as
\BEQ
\label{eq-invest}
\begin{array}{ll}
\mbox{minimize} & \sum_{i=1}^n f(x_i) \\
\mbox{subject to} & A x \leq b \\
& 0 \leq x.
\end{array}
\EEQ

We let
\[
f(x) = \left\{
        \begin{array}{ll}
            1 & \quad 0 \leq x < 1 \\
            0 & \quad x \geq 1.
        \end{array}
    \right. 
\]
Random instances of the investment problem are generated 
with $n$ variables and $m$ constraints.
Random sector constraints are generated by 
choosing entries of $A$ to be 
0 or 1 uniformly at random with probability 1/2,
and let $b = 1/2 A\ones$, where $\ones$ is the vector of all ones,
in order to ensure the constraints are binding.

The results of our numerical experiments are presented in Table~\ref{t-invest}
and Figure~\ref{f-invest}.
In the table, we choose $n=50$, $m=10$, 
let $\hat{x}$ be the solution to the problem
\BEQ
\label{eq-invest-conv}
\begin{array}{ll}
\mbox{minimize} & \sum_{i=1}^n \hat f(x_i) \\
\mbox{subject to} & A x \leq b \\
& 0 \leq x
\end{array}
\EEQ
returned by an interior point solver, and let $x^\star$ be the 
solution to the random LP \ref{eq-lp}.
The observed difference between $f(x^\star)$ and $p^\star$ is always substantially smaller than
the theoretical bound of $m \rho(f) = 10$.

Figure~\ref{f-invest} shows the improvement from solving \ref{eq-lp}, calculated as $\frac{f(x^\star) - f(\hat x)}{f(x^\star) - p^\star}$,
as a function of the number of variables $n$ and constraints $m$,
averaged over 10 random instances of the problem.
Solving the random LP \ref{eq-lp} gives a substantial improvement
when $m < n$.

\begin{table}[htb!]
\centering
\caption{\label{t-invest}
Investment problem.}
\begin{tabular}{l|l|l|l|l}
$f(x^\star)$ & $f(\hat{x})$ & $p^\star$ & $\hat{p}$ & \% improved\\
\hline
43.01 & 23.01 & 22.00 & 20.25 & 0.95\\
29.02 & 26.00 & 22.00 & 20.36 & 0.43\\
30.09 & 24.00 & 21.00 & 19.92 & 0.67\\
26.32 & 25.00 & 22.00 & 20.27 & 0.31\\
24.68 & 24.00 & 22.00 & 20.33 & 0.25\\
26.01 & 25.00 & 21.00 & 19.26 & 0.20\\
26.46 & 24.00 & 20.00 & 19.40 & 0.38\\
28.24 & 25.00 & 23.00 & 20.65 & 0.62\\
29.04 & 24.00 & 21.00 & 20.21 & 0.63\\
27.01 & 23.01 & 21.00 & 19.70 & 0.67\\
\end{tabular}
\end{table}

\begin{figure}
\centering

\begin{tikzpicture}
  \begin{axis}[view={0}{90},
               width=8cm,
               colorbar,
               colormap/blackwhite,
               xlabel=$m$,
               ylabel=$n$,]
    \addplot3[surf,domain=0:100,y domain=0:100] table {\home/invest2.tex};
  \end{axis}
\end{tikzpicture}

\caption{\label{f-invest}Improvement $\frac{f(x^\star) - f(\hat x)}{f(x^\star) - p^\star}$ on investment problem.}
\end{figure}
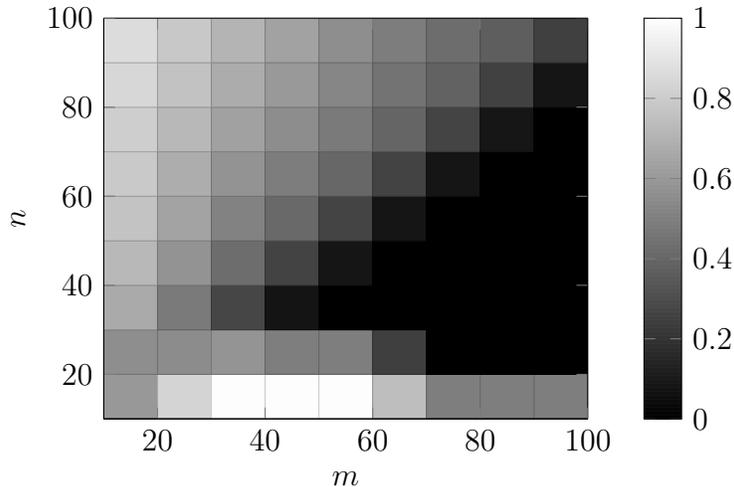



\subsection{Solution via ADMM}

Here we demonstrate how to use ADMM, a framework for distributed optimization,
to find $\hat{x}$ satisfying the bound on the duality gap.
This shows that a solution satisfying the bound may be found
even for very large scale problems, 
so long as the proximal operators of the functions $f_i$ can be evaluated efficiently.


\paragraph{ADMM.} The Alternating Directions Method of Multipliers (ADMM)
was introduced in 1975 by Glowinski and Marocco \cite{GlM:75} 
and Gabay and Mercier \cite{GaM:76},
and is closely related to a number of classical operator-splitting methods such as Douglas-Rachford and Peaceman-Rachford \cite{Gab:83,PR:55,LM:79,glowinski2014}.
ADMM has recently received renewed interest as a method for solving distributed optimization problems
due both to its ease of implementation and its robust convergence in practice
and in theory on convex problems \cite{Gab:83,FoG:83,FG:83,GLT:87,Tse:91,Fuk:92,EB:92,EF:93,CT:94,hong2012,he2012}.
For an introduction to ADMM, we refer the reader to the survey \cite{boyd2011} and references therein.

ADMM is not guaranteed to converge to the global solution when applied to a nonconvex problem
\cite{zhang2010,magnusson2014}.
However, its computational advantages still make ADMM a popular method 
for nonconvex optimization \cite{magnusson2014,derbinsky2013,chartrand2013,chartrand2012,bouaziz2013,gilimyanov2013,kanamori2012} even in the absence of convergence guarantees.
In constrast to this previous work, here we use ADMM to find a feasible point for the 
nonconvex problem which obeys the global error bound of Theorem~\ref{thm-bound}.

\paragraph{ADMM for the convexified problem.}
A generalized consensus ADMM iteration can be used to solve the convexified problem.
(See \cite{boyd2011} for details.)
We rewrite the problem as
\BEQ
\label{eq-primal-conv-consensus}
\begin{array}{ll}
\mbox{minimize} & \sum_{i=1}^n \hat{f}_i(x_i) + \ones_{Ax \leq b, Gx=h}(z)\\
\mbox{subject to} & x=z,
\end{array}
\EEQ
where $\ones_\mathcal{C}$ denotes the indicator function of the set $\mathcal C$.
An ADMM iteration solving the above problem is given by
\BEAS
x^k_i &=& \argmin \hat{f}_i(x) + \rho/2 \|x - z^{k-1}_i + y^{k-1}_i\|_2^2 \\
z^k &=& \Pi_{Ax \leq b, Gx=h} (x^k) \\
y^k_i &=& y^{k-1}_i + 1/\rho(x^k_i - z^k_i). 
\EEAS
Here, $\Pi_\mathcal{C}$ denotes projection onto the set $\mathcal C$,
and $\rho>0$ is a parameter.
Under some mild conditions \cite{hong2012},
the iterates $z^k$ and $x^k$ both converge linearly to
a primal optimal solution $x^\star$ for the convexified problem; 
$y^k$ converges to a dual optimal solution $\lambda^\star$ for the convexified problem.

This iteration requires very little communication between
nodes in a distributed system. This property may be very useful if it is
expensive to compute or to optimize the convex envelopes $\hat{f}_i$.
Each processor in the distributed architecture may perform the $x$ update 
for one block $i$ in parallel, with no need to communicate with other processors.
The only centralized computation is the projection of $x^k$ onto the constraints. 

However, we have already seen in \S\ref{s-counterexample}
that projecting a solution to the dual problem
onto the constraint set can work very poorly for nonconvex separable problems.
To understand this phenomenon better, consider a \emph{symmetric} problem,
which has the same $f_i$ for every $i=1,\ldots,n$, 
and constraint matrices $A$ and $G$ whose columns are identical.
ADMM will not break the symmetry between different coordinate blocks,
since the iteration above is completely symmetric, 
resulting in a symmetric solution to the convexified problem.
But we have seen in \S\ref{s-counterexample} that a symmetric solution
is the worst sort of solution; it can have an error that grows linearly with $n$.

\paragraph{ADMM for the randomized problem.}
We want a solution at an extreme point of the optimal set for the convexified
problem. Fortunately, it is also easy to compute the solution to
the randomized problem \ref{eq-lp}
using distributed optimization, which allows us to find a point $\hat{x}$
satisfying the bound in Theorem~\ref{thm-bound}.

Taking the primal and dual optimal pair $(x^\star,\lambda^\star)$
for (\ref{eq-primal-conv-consensus}) computed by the first round of ADMM iterations, we 
can rewrite problem~(\ref{eq-lp-dual}) in ADMM consensus form.
Let 
\[
M = \left\{x: \hat{f}(x)+ \lambda^{^\star T} A x \leq \hat{f}(x^\star)+ \lambda^{^\star T} A x^\star \right\}.
\]
We saw in \S\ref{s-proofs} that $M$ is separable, 
and can be written as $M = M_1 \times \cdots \times M_n$.
Hence we can rewrite \ref{eq-lp} as
\[
\label{eq-lp-consensus}
\begin{array}{ll}
\mbox{minimize} & \sum_{i=1}^n (w_i^T x_i + \ones_{M_i}(x_i)) 
+ \ones_{Ax \leq b, Gx=h}(z)\\
\mbox{subject to} & x=z,
\end{array}
\]
which gives rise to the ADMM consensus iteration
\BEAS
x^k_i &=& \argmin_{x \in M_i} w^T_i x + \rho/2 \|x - z^{k-1}_i + y^{k-1}_i\|_2^2 \\
z^k &=& \Pi_{Ax \leq b, Gx=h} (x^k) \\
y^k_i &=& y^{k-1}_i + 1/\rho(x^k_i - z^k_i). 
\EEAS
The solution $z$ produced by this distributed iteration will satisfy Theorem~\ref{thm-bound}.


\section*{Acknowledgements}
The authors thank Haitham Hindi, Ernest Ryu
and the anonymous reviewers 
for their very careful readings of and comments on early drafts of this paper,
and Jon Borwein and Julian Revalski for their generous advice
on the technical lemmas in the appendix.

\appendix
\section{The dual of the dual is the convexified problem}\label{s-dual-dual}

In this appendix, we prove that the dual of the dual of \ref{eq-primal}
is the convexified problem \ref{eq-primal-conv}.

Before we begin, note that the convex envelope has a close connection to duality.
Let $f^*(y) = \sup (y^T x - f(x)) = - \inf (f(x) - y^T x)$ 
be the (Fenchel) \emph{conjugate} of $f$.
Then $\hat f(x) = f^{**}(x)$ is the \emph{biconjugate} of $f$ \cite{rockafellar1970}.
The conjugate function arises naturally when taking the dual of a problem,
as we show below.
Hence it should come as no suprise that the biconjugate appears upon
taking the dual twice.

Below, 
we refer to the dual of the dual problem as the dual dual problem,
the dual function of the dual problem as the dual dual function,
and the variables in the dual dual problem as the dual dual variables.

Recall the primal problem, which we write as
\[
\begin{array}{ll}
\mbox{minimize} & f(x) = \sum_{i=1}^n f_i(x_i) \\
\mbox{subject to} & A x \leq b \\
& G x = h.
\end{array}
\]

We can write the Lagrangian of the primal problem as
\[
L (x,\lambda,\mu) = \sum_{i=1}^n f_i(x_i) + \lambda^T (A x - b) + 
\mu^T(G x - h),
\] 
with dual variables $\lambda \geq 0$ and $\mu$.
The dual function $g(\lambda, \mu)$ is the minimum of the Lagrangian over $x$,
\BEAS
g(\lambda, \mu) &=& \inf_x L (x,\lambda,\mu) \\
&=& \inf_x \sum_{i=1}^n f_i(x_i) + \lambda^T (A x - b) + 
\mu^T(G x - h) \\
&=& \sum_{i=1}^n \inf_{x_i} (f_i(x_i) - \gamma_i x_i) - \lambda^T b - \mu^T h \\
&=& \sum_{i=1}^n -f_i^*(\gamma_i) - \lambda^T b - \mu^T h, \\
\EEAS
where we have defined $\gamma = - A^T \lambda - G^T \mu$ in the second to last equality
and used the relation $f^*(y) = - \inf (f(x) - y^T x)$ in the last.

The dual problem is to maximize the dual function over $\mu$ and $\lambda$
with $\lambda \geq 0$:
\[
\begin{array}{ll}
\mbox{maximize} & \sum_{i=1}^n -f_i^*(\gamma_i) - \lambda^T b - \mu^T h\\
\mbox{subject to} & \gamma = - A^T \lambda - G^T \mu \\
& \lambda \geq 0.
\end{array}
\]
The conjugate function $f_i^*$ is a pointwise supremum of affine functions,
and so is always convex even if $f_i$ is not.
Hence the dual problem is a concave maximization problem. 

To take the dual of the dual, we perform exactly the same computations again
on the dual problem now instead of the primal.
The dual Lagrangian is 
\[
L_D (\lambda,\mu,\gamma,x,y) = \sum_{i=1}^n -f_i^*(\gamma_i) - \lambda^T b - \mu^T h
+ x^T(\gamma + A^T \lambda + G^T) + s^T \lambda,
\] 
with dual dual variables $s \geq 0$ and $x$.
We maximize the dual Lagrangian over the dual variables $\lambda$, $\mu$, and $\gamma$
to form the dual dual function
\BEAS
g_D(x, s) &=& \sup_{\lambda \geq 0, \mu, \gamma} L_D (\lambda,\mu,\gamma,x,y) \\
&=& \sup_{\lambda \geq 0,\mu,\gamma} \sum_{i=1}^n -f_i^*(\gamma_i) 
- \lambda^T b - \mu^T h
+ x^T(\gamma + A^T \lambda + G^T) + s^T \lambda\\
&=& \sup_{\lambda \geq 0,\mu} \sum_{i=1}^n f_i^{**}(x_i) 
+ \lambda^T (Ax + s - b) + \mu^T (Gx - h),
\EEAS
using now the relation $f^*(y) = \sup (y^T x - f(x))$.
This is finite only if $Ax + s - b \leq 0$ and $Gx - h = 0$.
So we see
\[
g_D(x, s) = \sum_{i=1}^n f_i^{**}(x_i)
\]
so long as these equalities are satisfied.

To form the dual dual problem, we minimize the dual dual function over
$x$ and $s \geq 0$: 
\[
\begin{array}{ll}
\mbox{minimize} & \sum_{i=1}^n f_i^{**}(x_i) \\
\mbox{subject to} & A x \leq b \\
& G x = h,
\end{array}
\]
where we have solved for $s = b - Ax$.
Hence we see that we have recovered the convexified problem 
by dualizing the primal twice.
\clearpage

\section{Well-posedness}
\label{s-well-posed}

The following theorem characterizes the set of vectors in the dual space for which 
linear optimization over a compact set $S$ is well-posed.

\begin{theorem}[Well-posedness of linear optimization]
\label{thm-well-posed}
Suppose $S$ is a compact set in $\reals^n$. Then the set of $w\in \reals^n$ for which 
the maximizer of $w^Tx$ over $S$ is not unique has (Lebesgue) measure zero.
\end{theorem}

This result is well-known; for example, it follows from \cite[\S2]{bolte2011},
taking into account that if $S \subseteq \reals^n$ is compact, then so is its
convex hull $K = \conv(S)$ and the set of extreme points of S and K coincide.
In fact, one can derive much stronger results using, for example, Alexandrov's
theorem for convex functions to show quadratic decay, or finite identifiability 
in the case of semialgebraic functions.
However, our purpose here is more modest; we merely prove the weaker result stated
as Theorem \ref{thm-well-posed} so that this paper may be self-contained.

Before proceeding to a proof, however, let us make sense of the statement of the theorem.
By definition, the maximizer of a linear functional over a set $S$ is a 
face $R$ of $S$. The maximizer is unique if and only if $R$ is a zero-dimensional
face (\ie, an extreme point).
Only an outward normal to a face will be maximized on that face. 

It is easy to see that the theorem is true for polyhedral sets $S$. 
For each face of the polyhedron that is not extreme, the set of vectors maximized by that face
(the set of outward normals to the face, \ie, the normal cone)
will have dimension \emph{smaller than} $n$.
A polyhedron has only a bounded number of faces, so the union of these sets
still has measure zero.

On the opposite extreme, consider the unit sphere. 
A sphere has an infinite number of faces.
But every face is extreme, and every vector $w$ has a unique maximizer.

The difficulty comes when we consider cylindrical sets:
those constructed as the Cartesian product of a sphere and a cube.
Here, every outward normal to the ``sides'' of the cylinder is a vector
whose maximum over the set is not extreme.
That is, we find an \emph{uncountably infinite} number of faces (parametrized by the 
boundary of the sphere)
that are not extreme points.

\begin{proof}
Let $I_S: \reals^n \to \reals$ be the indicator function of $S$.
$S$ is compact, so the convex conjugate $I^*_S(y) = \sup_x y^Tx - I_S(x)$
of $I_S$ is finite for every $y \in \reals^n$.
Rachemacher's Theorem \cite[Theorem 2.5.1]{borwein2010} states that 
a convex function $g:\reals^n \to \reals$ 
is differentiable almost everywhere with respect to Lebesgue measure on $\reals^n$.
Furthermore, if $I^*_S$ is differentiable at $y$ with $\nabla I^*_S(y) = x$, then 
$y^Tx - I_S(x)$ attains a strong maximum at $x$ \cite[Theorem 5.2.3]{borwein2010};
that is, there is a unique maximizer of $y^Tx$ over $S$.
\end{proof}

Clearly, the statement also holds for the minimizers, rather than maximizers, of $w^Tx$.

The following corollary will be used in the proof of the main theorem of this paper.

\begin{corollary}
\label{thm-well-posed-linear-compact}
Suppose $S$ is a compact set in $\reals^n$, and $w$ is a uniform random
variable on the unit sphere in $\reals^n$. Then with probability one, there
is a unique minimizer of $w^Tx$ over $S$.
\end{corollary}

\begin{proof}
The property of having a unique minimizer exhibits a symmetry along radial lines:
there is a unique minimizer of $w^Tx$ over $S$ if and only if
there is a unique minimizer of $(w/\|w\|_2)^Tx$ over $S$.
A uniform random vector on the unit sphere may be generated by taking a uniform random vector 
on the unit ball, and normalizing it to lie on the unit sphere.
Since the set of directions whose maximizers are not unique 
has Lebesgue measure zero,
the vectors on the unit sphere generated in this manner 
have maximizers that are unique with probability one.
\end{proof}

We give one last corollary, which may be of mathematical interest,
but is not used elsewhere in this paper.

\begin{corollary}
\label{cor-normal-cones}
Suppose $S$ is a compact set in $\reals^n$.
The union of the normal cones $N(x)$ of
all points $x \in S$ that are not extreme has measure zero.
\end{corollary}

\begin{proof}
A point $x$ minimizes $y^Tx$ over $S$ if and only if $y \in N(x)$.
A point $x$ is the only minimizer of $y^Tx$ over $S$ if and only if $x$ is exposed,
and hence extreme.
Hence no $y$ with a unique minimizer over $S$ lies in the normal cone of a point
that is not extreme.
Thus the union of the normal cones $N(x)$ of
all points $x \in S$ that are not extreme is a
subset of the vectors which do not have a unique maximizer over $S$,
and hence has measure zero.
\end{proof}

\clearpage
\section{Closure}
\label{s-closed}


The following lemma techical lemma will be useful in the main body of the paper.

\begin{lemma}
\label{thm-closed}
Let $S \subset \reals^n$ be a nonempty compact set, and 
let $f: S \to \reals$ be lower semi-continuous on $S$.
Then $\conv (\epi f)$ is closed.
\end{lemma}

This result follows from \cite[Thm. 4.6]{benoist1996}, 
since every function defined on a compact set is in particular 1-coercive.
The earliest proof known to the authors can be found in \cite[p. 69]{valadier1970};
for a simpler exposition, see \cite[Ch. X, \S1.5]{hiriart1996}.
\begin{boringproof}
Here, we provide a self-contained elementary proof for the curious reader.


\begin{proof}

Every point $(x,t) \in \cl (\conv (\epi f))$ is a limit of points 
$(x^k, t^k)$ in $\conv (\epi f)$.
These points can be written as
\[
(x^k, t^k) = \sum_{i=1}^{n+2} \lambda^k_i (a^k_i,b^k_i)
\]
with $\sum_{i=1}^{n+2} \lambda^k_i = 1$, $0 \leq \lambda^k_i \leq 1$,
and $(a^k_i,b^k_i) \in \epi(f)$.
Since $[0,1]$ and $S$ are compact, we can find a subsequence along which
each sequence $a^k_i$ converges to a limit $a_i \in S$, and 
each sequence $\lambda^k_i$ converges to a limit $\lambda_i \in [0,1]$.
 
Let $P = \{i : \lambda_i > 0$\}.  
Note that $P$ is not empty, since 
$\sum_{i=1}^{n+2} \lambda^k_i = 1$ for every $k$.
If $l \in P$, then because the limit $t$ exists, $\limsup_k b_i^k$ is bounded above.  
Recall that a lower semi-continuous function is bounded below on a compact domain,
so $b_i^k$ is also bounded below. 
This shows that for $i \in P$, every subsequence of $b_i^k$ has a subsequence that converges to a limit $b_i$.
In particular, we can pick a subsequence $k_j$ 
such that simultaneously, for $i=1,\ldots,n+2$, $a_i^{k_j}$, $b_i^{k_j}$, and $\lambda_i^{k_j}$
converge along the subsequence $k_j$ to $a_i$, $b_i$, and $\lambda_i$, respectively.

Define $S_P = \sum_{i\in P} \lambda_i b_i$.
Then along the subsequence $k_j$, $\lim_{j \to \infty} \sum_{i\notin P} \lambda_i^{k_j} b_i^{k_j} = t - S_P$ also exists.
Since $f$ is bounded below, $b_i^k$ are all bounded below, 
and for $i \notin P$, $\lambda_i^k \to 0$, 
so $t- S_P \ge 0$.
Therefore $(x,t)$ can be written as $\sum_{i \in P} \lambda_i (a_i,b_i) + (0,  t- S_P)$.

Recall that a function is lower semi-continuous if and only if its epigraph is closed.
Hence $(a_i, b_i) \in \epi f$ for $i \in P$.
Without loss of generality, suppose $1 \in P$,
and note that $(a_1, b_1 + t - S_P) \in \epi f$, since $t-S_P$ is non-negative.

Armed with these facts, we see we can write $(x,t)$ as a convex combination
of points in $\epi f$,
\[
(x,t) = \lambda_1 (a_1, b_1  + t-S_P) + \sum_{i\in S, i \ne 1} \lambda_i (a_i,b_i).
\]
Thus every $(x,t) \in \cl (\conv (\epi f))$ can be written as 
a convex combination of points in $\epi f$, so $\conv (\epi f)$ is closed.
\end{proof}

\end{boringproof}

\begin{corollary}
Let $S \subset \reals^n$ be a compact set, and 
let $f: S \to \reals$ be lower semi-continuous on $S$.
Then $\epi(\hat{f}) = \cl( \conv (\epi f)) = \conv (\epi f)$.
\end{corollary}




\bibliography{\home/db}
\end{document}